\documentclass{amsart}
\usepackage{amssymb, amsmath}
\newcommand{\g}[2]{\ensuremath{\langle #1,#2 \rangle}}
\newcommand{\ja}[1]{\ensuremath{\mathcal{J}( #1)}}
\newcommand{\model}{$(V,\langle\cdot,\cdot \rangle, A)$ }

\usepackage{color}

\newtheorem{theorem}{Theorem}
\newtheorem{lemma}[theorem]{Lemma}

\begin{document}

\title[The Osserman condition and the Raki\'c duality principle]{Equivalence between the Osserman condition and the Raki\'c duality principle in dimension four}
\author{M. Brozos-V\'{a}zquez $\,$ E. Merino}
\address{Department of Mathematics, University of A Coru\~na, Spain}
\email{miguel.brozos.vazquez@udc.es, eugenio.merino@udc.es}
\thanks{
M. B.-V. is supported by projects MTM2009-07756 and INCITE09 207 151 PR (Spain). E. M. is supported by projects MTM2008-05861 and INCITE09 207 151 PR (Spain)}
%\subjclass{53C21, 53C50, 53C25}

\begin{abstract}
We show that $4$-dimensional Riemannian manifolds which satisfy the Raki\'c duality principle are Osserman (i.e. the eigenvalues of the Jacobi operator are constant), thus both conditions are equivalent.
\end{abstract}

\maketitle

%\section{Introduction}

The Jacobi operator of a two-point-homogeneous Riemannian manifold has constant eigenvalues. In \cite{Osserman} Osserman wondered if the converse is true. That problem, known in the literature as the Osserman problem, was solved by the contributions of several authors, see \cite{Chi1,Gilkey-Swann-Vanhecke,Ni1,Ni2} and \cite{GKV, Gilkey-book} for a broad exposition on the topic.

More formally the Osserman condition can be phrased as:

\noindent{\bf Pointwise Osserman condition.} A Riemannian manifold $(M,g)$ is pointwise Osserman if the eigenvalues of the Jacobi operator $\ja{x}=R(\cdot,x)x$ do not depend on the unit vector $x\in T_pM$, for every point $p\in M$ (the eigenvalues may vary from point to point).

On the process of studying Osserman manifolds, it was shown that pointwise Osserman manifolds satisfy the following duality principle (see \cite{R}, \cite{gilkey-p-osserman} and \cite{Gilkey-book}), which has also been investigated recently in higher signature in \cite{AR}.

\noindent{\bf Raki\'c duality principle.} A Riemannian manifold $(M,g)$ satisfies the Raki\'c duality principle if for every point $p\in M$ and for any unit vectors $x,y\in T_pM$:
\begin{equation}\label{eq:duality}
\ja{x}y=\lambda y\quad \Rightarrow\quad \ja{y}x=\lambda x\,,
\end{equation}
where $\lambda$ is a real number.

Note that both definitions are pointwise. The pointwise Osserman condition is not equivalent to the global one (i.e. the eigenvalues do not depend on the point $p$) in dimension four \cite{Gilkey-Swann-Vanhecke}. This fact, together with some other features that we will recall in Section~\ref{section:preliminaries}, makes dimension four a special dimension for the Osserman problem. It is an open problem in this context whether the Raki\'c duality principle implies the Osserman condition. The following is the main theorem of the paper and provides an affirmative answer to that question.

\begin{theorem}\label{th:main}
Let $(M,g)$ be a Riemannian manifold of dimension $4$. The following assertions are equivalent:
\begin{enumerate}
\item[(i)] $(M,g)$ is pointwise Osserman.
\item[(ii)] $(M,g)$ satisfies the Raki\'c duality principle.
\end{enumerate}
\end{theorem}

{\bf Outline of the paper.} In Section~\ref{section:preliminaries} we recall some results and introduce the notation we will use in the proof of Theorem~\ref{th:main}. In Section~\ref{section:dim3} we show that in dimension three the Raki\'c duality principle implies the Osserman condition, thus showing that an analogous of Theorem~\ref{th:main} is also true in a lower dimension. Finally, in Section~\ref{section:dim4}, we prove Theorem~\ref{th:main}.

\section{Preliminaries}\label{section:preliminaries}
We work at a purely algebraic level. Let $V$ be a vector space of dimension $n$, $\g{\cdot}{\cdot}$ a positive definite inner product and $A$ an algebraic curvature tensor, i.e. a $(0,4)$-tensor which satisfies the following relations:
\begin{equation}\label{eq:sim-curv}
\begin{array}{l}
A(x,y,z,w)=-A(y,x,z,w)=A(z,w,x,y),\\
A(x,y,z,w)+A(y,z,x,w)+A(z,x,y,w)=0\,.
\end{array}
\end{equation}
We refer to the triple $(V,\langle\cdot,\cdot \rangle, A)$ as an algebraic model.

We use the inner product to upper indices and define the curvature operator by $\g{A(x,y)z}{w}:=A(x,y,z,w)$ for any vectors $x,y,z,w\in V$. Thus the Jacobi operator is defined as $\mathcal{J}(x)y:=A(y,x)x$. Note that $\ja{x}x=0$, so we restrict $\ja{x}$ to $x^\perp$ henceforth.

Using analogy with the pointwise geometric setting, we say that the model \model is Osserman if the eigenvalues of $\ja{x}$ do not depend on the unit vector $x\in V$. Similarly, we say that \model satisfies the Raki\'c duality principle if for any eigenvalue $\lambda$ we have that
$\ja{x}y=\lambda y$ if and only if $\ja{y}x=\lambda x$.

A model is said to be Einstein if $\rho(\cdot,\cdot)=c\g{\cdot}{\cdot}$, where $\rho(x,y):=Tr\{z\rightarrow A(z,x)y\}$ is the Ricci tensor and $c$ is a real number. Contract this identity to see that $c=\frac{\tau}{n}$, where $\tau$ denotes the scalar curvature. Every Osserman model is Einstein (see \cite{GKV,Gilkey-book}).

A particular feature of $4$-dimensional models is that the Hodge star operator $\star$ acts on the space of bi-vectors $\Lambda=\{x\wedge y: x,y\in V\}$ satisfying $\star^2=\operatorname{Id}$, where $\operatorname{Id}$ stands for the identity map. This induces a splitting $\Lambda=\Lambda^+\oplus\Lambda^-$ into the eigenspaces associated to the $+1$ and $-1$ eigenvalues.
For an orthonormal basis $\{e_1,e_2,e_3,e_4\}$ of $V$, an orthonormal basis of $\Lambda^\pm$ is given by
\begin{equation}\label{eq:basis-bivectors}
\begin{array}{c}
\Lambda^\pm=\operatorname{span}\left\{E_1^\pm=\frac{e^1\wedge e^2\pm e^3\wedge e^4}{\sqrt{2}},E_2^\pm=\frac{e^1\wedge
e^3\mp e^2\wedge e^4}{\sqrt{2}},E_3^\pm=\frac{e^1\wedge e^4\pm e^2\wedge e^3}{\sqrt{2}}\right\}.
\end{array}
\end{equation}
The Weyl tensor in dimension four is given by
\[
\begin{array}{rcl}
W(x,y,z,w)&\!=&\!R(x,y,z,w)+\frac{\tau}6\{\g{x}{w}\g{y}{z}-\g{x}{z}\g{y}{w}\}\\
\noalign{\smallskip}
&&\!-\frac12 \{\rho(x,w)\g{y}{z}+\rho(y,z)\g{x}{w}-\rho(x,z)\g{y}{w}-\rho(y,w)\g{x}{z}\}.
\end{array}
\]
Denote by $W^\pm$ the restriction of the Weyl tensor acting on bi-vectors to $\Lambda^\pm$. A model \model is said to be self-dual if $W^-=0$ and anti-self-dual if $W^+=0$.

In Section~\ref{section:dim4} we will use the following well-known characterization of Osserman models in dimension four \cite{Gilkey-Swann-Vanhecke}.

\begin{theorem}\label{th:einstein-self-dual}
A model \model of dimension $4$ is Osserman if and only if it is Einstein and self-dual (or anti-self-dual).
\end{theorem}

\section{The Raki\'c duality principle in dimension $3$}\label{section:dim3}

We begin by showing that the Raki\'c duality principle implies the Osserman condition for an algebraic model \model of dimension $3$; thus we have the following:

\begin{theorem}
Let \model be a $3$-dimensional algebraic model. The following two conditions are equivalent:
\begin{enumerate}
\item[(i)] \model is Osserman.
\item[(ii)] \model satisfies the Raki\'c duality principle.
\end{enumerate}
\end{theorem}
\begin{proof}
That (i) implies (ii) was proved in \cite{R}. We assume that \model satisfies the duality in \eqref{eq:duality} to prove the converse.

Let $x$ be a unit vector, $\ja{x}$ is self-adjoint and hence diagonalizable. Denote by $\lambda$ and $\mu$ the eigenvalues of the Jacobi operator $\ja{x}$ restricted to $x^\perp$. Let $y$ and $z$ be unit eigenvectors associated to $\lambda$ and $\mu$, respectively, so that $\ja{x}y=\lambda y$ and $\ja{x}z=\mu z$. By the duality principle we also have
\[
\ja{y}x=\lambda x\qquad \text{and}\qquad \ja{z}x=\mu x\,.
\]
Since $\ja{y}$ preserves the subspace $\operatorname{span}\{x,y\}$, it also preserves $\operatorname{span}\{x,y\}^\perp$, so $\ja{y}z=\gamma z$ and, by duality, $\ja{z}y=\gamma y$, for a certain eigenvalue $\gamma$.

Compute
\[
\ja{\cos\theta x+\sin\theta y}(-\sin\theta x+\cos\theta y)=\lambda (-\sin\theta x+\cos\theta y)
\]
to see that $\ja{\cos\theta x+\sin\theta y}$ leaves  $\operatorname{span}\{x,y\}$ invariant. So $z$ is an eigenvector for $\ja{\cos\theta x+\sin\theta y}$ and there exists $\alpha$ such that $\ja{\cos\theta x+\sin\theta y}z=\alpha z$. Hence, by \eqref{eq:duality}, we have:
\[
\alpha(\cos\theta x+\sin\theta y)=\ja{z}(\cos\theta x+\sin\theta y)=\cos\theta \mu x+\sin \theta \gamma y\,.
\]
Since $x$ and $y$ are linearly independent we get that $\alpha=\mu=\gamma$.
We repeat the same argument for $\cos\theta y+\sin\theta z$ to see that, indeed, $\lambda=\mu=\gamma$.

Since $x$ was chosen arbitrarily we have just shown that $\ja{x}=\lambda(x) \operatorname{Id}$ for every unit vector $x\in V$, where $\lambda(x)$ is a function of $x$. In order to finish the proof we must show that $\lambda$ is constant. Take $x$ and $y$ unit vectors. There exists $z\perp x,y$ so that $\ja{x}z=\lambda(x)z$ and $\ja{y}z=\lambda(y)z$. By the Raki\'c duality principle we have that $\ja{z}x=\lambda(z)x=\lambda(x)x$ and that $\ja{z}y=\lambda(z)y=\lambda(y)y$. Hence $\lambda(x)=\lambda(z)=\lambda(y)$ and \model is Osserman.
\end{proof}

\section{Proof of Theorem~\ref{th:main}}\label{section:dim4}

Theorem~\ref{th:main} will be a consequence of the following sequence of lemmas. We begin by choosing an appropriate basis for our subsequent analysis.

\begin{lemma}\label{lemma:4-dim-basis}
Let \model be a $4$-dimensional algebraic model which satisfies the Raki\'c duality principle. Then every unit vector $x\in V$ can be completed to an orthonormal basis $\{x,y,z,w\}$ such that
\begin{equation}\label{eq:eigenvalues}
\begin{array}{c}
\ja{x}y=\lambda_1 y,\qquad \ja{x}z=\lambda_2 z,\qquad \ja{x}w=\lambda_3 w,\\
\ja{y}z=\lambda_4 z,\qquad \ja{y}w=\lambda_5 w,\qquad \ja{z}w=\lambda_6 w,
\end{array}
\end{equation}
where $\lambda_1$, $\lambda_2$, $\lambda_3$, $\lambda_4$, $\lambda_5$ and $\lambda_6$ are real numbers.
\end{lemma}
\begin{proof}
Let $x$ be a unit vector and let $y$, $z$, $w$ be a basis of unit eigenvectors associated to $\ja{x}$, i.e.
\[
\ja{x}y=\lambda_1 y,\qquad \ja{x}z=\lambda_2 z,\qquad \ja{x}w=\lambda_3 w\,.
\]
Since $\ja{y}x=\lambda_1 x$ and $\ja{y}y=0$, we have that $\ja{y}span\{z,w\}\subset span\{z,w\}$. Hence there exist $a,b\in \mathbb{R}$ such that $a^2+b^2=1$ and
\[
\ja{y}(az+bw)=\alpha (az+bw)\,,\qquad \ja{y}(bz-aw)=\beta (bz-aw)\,.
\]
By the Raki\'c duality principle
\[
\begin{array}{rcl}
\alpha y&=&\ja{az+bw}(y)=a^2\ja{z}y+b^2\ja{w}y+ab\{A(y,z)w+A(y,w)z\},\\
\noalign{\smallskip}
\beta y&=&\ja{bz-aw}(y)=b^2\ja{z}y+a^2\ja{w}y-ab\{A(y,z)w+A(y,w)z\}.
\end{array}
\]
Sum both equation to get
\[
(\alpha+\beta)y=a^2\ja{z}y+b^2\ja{w}y+b^2\ja{z}y+a^2\ja{w}y=\ja{z}y+\ja{w}y\,.
\]
We already had that $\g{\ja{z}y}{x}=\g{\ja{w}y}{x}=0$ because $x$ is an eigenvector for $\ja{z}$ and $\ja{w}$, and these are self-adjoint. Take the inner product of the previous expression with respect to $z$ and $w$ to see that $\g{\ja{z}y}{w}=\g{\ja{w}y}{z}=0$. Therefore, $\ja{z}y=\lambda_4 y$ and $\ja{w}y=\lambda_5 y$ for certain $\lambda_4,\lambda_5\in \mathbb{R}$. Now the Raki\'c duality principle implies that $\ja{y}z=\lambda_4 z$ and $\ja{y}w=\lambda_5 w$. Finally, since $x$ and $y$ are eigenvectors for $\ja{z}$, so is $w$ which generates the orthogonal complement of $\operatorname{span}\{x,y\}$ in $z^\perp$. Hence $\ja{z}w=\lambda_6 w$ for a certain $\lambda_6\in \mathbb{R}$.
\end{proof}

\begin{lemma}\label{lemma:einstein}
Let \model be a $4$-dimensional model. If \model satisfies the Raki\'c duality principle then it is Einstein.
\end{lemma}
\begin{proof}
We adopt notation in Lemma~\ref{lemma:4-dim-basis}.
For any $\theta\in [0,2\pi)$, set $r_\theta=\cos \theta x+\sin\theta y$. Note that $\ja{r_\theta}(\sin\theta x-\cos\theta y)=\lambda_1 (\sin\theta x-\cos\theta y)$. Hence $\ja{r_\theta}\operatorname{span}\{z,w\}\subset \operatorname{span}\{z,w\}$ and there exist $a,b\in \mathbb{R}$ with $a^2+b^2=1$ such that
\[
\ja{r_\theta}(az+bw)=\alpha (az+bw),\quad \ja{r_\theta}(bz-aw)=\beta (bz-aw),
\]
for certain $\alpha$ and $\beta$ real numbers.
Expand
\begin{equation}\label{eq:expand-jacobi}
\ja{r_\theta}=\cos^2 \theta\ja{x}+\sin^2\theta\ja{y}+\cos\theta \sin\theta\{A(\cdot,x)y+A(\cdot,y)x\}\,.
\end{equation}
Now, we particularize $\theta=\frac{\pi}{4}$ to set $r=r_\frac{\pi}4=\frac{x+y}{\sqrt{2}}$. Using expressions in \eqref{eq:eigenvalues} and equation \eqref{eq:expand-jacobi} we compute
\[
\begin{array}{rcl}
\ja{r}(az+bw)&=&\frac{a}2\{ \lambda_2z+\lambda_4z+A(z,x)y+A(z,y)x\}\\
\noalign{\smallskip}
&&+\frac{b}2\{ \lambda_3w+\lambda_5w+A(w,x)y+A(w,y)x\}
\end{array}
\]
and
\[
\begin{array}{rcl}
\ja{r}(bz-aw)&=&\frac{b}2\{ \lambda_2z+\lambda_4z+A(z,x)y+A(z,y)x\}\\
\noalign{\smallskip}
&&-\frac{a}2\{ \lambda_3w+\lambda_5w+A(w,x)y+A(w,y)x\}.
\end{array}
\]
Taking the inner product with $z$ and $w$ we obtain the following equations:
\begin{equation}\label{eq:rel-rak1}
\begin{array}{rcl}
\alpha a&=&\frac{a}2\lambda_2+\frac{a}2\lambda_4+\frac{b}2\{A(w,x,y,z)+A(w,y,x,z)\},\\
\noalign{\smallskip}
\beta b&=&\frac{b}2\lambda_2+\frac{b}2\lambda_4-\frac{a}2\{A(w,x,y,z)+A(w,y,x,z)\},\\
\noalign{\smallskip}
\alpha b&=&\frac{b}2\lambda_3+\frac{b}2\lambda_5+\frac{a}2\{A(z,x,y,w)+A(z,y,x,w)\},\\
\noalign{\smallskip}
\beta a&=&\frac{a}2\lambda_3+\frac{a}2\lambda_5-\frac{b}2\{A(z,x,y,w)+A(z,y,x,w)\}.
\end{array}
\end{equation}
On the other hand, apply the Raki\'c duality principle to see that $\ja{az+bw}(x+y)=\alpha (x+y)$ and $\ja{bz-aw}(x+y)=\beta (x+y)$. Expanding we get
\[
\begin{array}{rcl}
\alpha(x+y)&=&a^2\lambda_2x+b^2\lambda_3 x+ab\{A(x,z)w+A(x,w)z\}\\
\noalign{\smallskip}
&&+a^2\lambda_4y+b^2\lambda_5 y+ab\{A(y,z)w+A(y,w)z\}
\end{array}
\]
and
\[
\begin{array}{rcl}
\beta(x+y)&=&b^2\lambda_2x+a^2\lambda_3 x-ab\{A(x,z)w+A(x,w)z\}\\
\noalign{\smallskip}
&&+b^2\lambda_4y+a^2\lambda_5 y-ab\{A(y,z)w+A(y,w)z\}.
\end{array}
\]
Sum both expressions to see that
\[
\lambda_2x+\lambda_3x+\lambda_4y+\lambda_5y=(\alpha+\beta)(x+y).
\]
Since $x$ and $y$ are linearly independent we obtain $\alpha+\beta=\lambda_2+\lambda_3=\lambda_4+\lambda_5$, so
\begin{equation}\label{eq:relation-1}
\lambda_2+\lambda_3-\lambda_4-\lambda_5=0.
\end{equation}
Take the inner product with $x$ and $y$ in the expressions above to obtain:
\begin{align}
\label{eq:array1} \alpha&=a^2\lambda_2+b^2\lambda_3+ab \{A(y,z,w,x)+A(y,w,z,x)\},\\ \label{eq:array2} \alpha&=a^2\lambda_4+b^2\lambda_5+ab \{A(x,z,w,y)+A(x,w,z,y)\},\\
\label{eq:array3} \beta&=b^2\lambda_2+a^2\lambda_3-ab \{A(y,z,w,x)+A(y,w,z,x)\},\\
\label{eq:array4} \beta&=b^2\lambda_4+a^2\lambda_5-ab \{A(x,z,w,y)+A(x,w,z,y)\}.
\end{align}
Compute \eqref{eq:array1}-\eqref{eq:array2}-\eqref{eq:array3}+\eqref{eq:array4} to obtain the following equation:
\begin{equation}\label{eq:relation-2}
(a^2-b^2)(\lambda_2-\lambda_4-\lambda_3+\lambda_5)=0.
\end{equation}
Now, from \eqref{eq:relation-1} and \eqref{eq:relation-2} we get two possibilities:
\begin{itemize}
\item $a^2=b^2$, which implies $\lambda_2=\lambda_5$ and $\lambda_3=\lambda_4$ by (\ref{eq:rel-rak1}), or
\item $\lambda_2=\lambda_4$ and $\lambda_3=\lambda_5$.
\end{itemize}

Repeat the previous argument interchanging $y$ by $z$ to see that:
\begin{itemize}
\item $\lambda_1=\lambda_4$ and $\lambda_3=\lambda_6$, or
\item $\lambda_1=\lambda_6$ and $\lambda_3=\lambda_4$,
\end{itemize}
and interchanging $y$ by $w$ to see that
\begin{itemize}
\item $\lambda_1=\lambda_5$ and $\lambda_2=\lambda_6$, or
\item $\lambda_1=\lambda_6$ and $\lambda_2=\lambda_5$.
\end{itemize}

In summary, combine the possibilities above to see that the possible eigenvalue structures are:
\begin{enumerate}
\item[a)] $\lambda_1=\lambda_2=\lambda_3=\lambda_4=\lambda_5=\lambda_6$,
\item[b)] $\lambda_1=\lambda_6$ and $\lambda_2=\lambda_3=\lambda_4=\lambda_5$,
\item[c)] $\lambda_2=\lambda_5$ and $\lambda_1=\lambda_3=\lambda_4=\lambda_6$,
\item[d)] $\lambda_3=\lambda_4$ and $\lambda_1=\lambda_2=\lambda_5=\lambda_6$,
\item[e)] $\lambda_1=\lambda_6$, $\lambda_2=\lambda_5$ and $\lambda_3=\lambda_4$.
\end{enumerate}
Now note that $\rho(u,v)=\g{\ja{x}u}{v}+\g{\ja{y}u}{v}+\g{\ja{z}u}{v}+\g{\ja{w}u}{v}=0$ for $u,v\in\{x,y,z,w\}$ with $u\neq v$. Also, since $\rho(v,v)=Tr\{\ja{v}\}$,  it is straightforward to see that the diagonal components of the Ricci tensor $\rho$ are given by
\begin{align*}
\rho(x,x)&=\lambda_1+\lambda_2+\lambda_3,&\rho(y,y)&=\lambda_1+\lambda_4+\lambda_5=\lambda_1+\lambda_2+\lambda_3,\\
\rho(z,z)&=\lambda_2+\lambda_4+\lambda_6=\lambda_1+\lambda_2+\lambda_3,&\rho(w,w)&=\lambda_3+\lambda_5+\lambda_6=\lambda_1+\lambda_2+\lambda_3.
\end{align*}
Hence $\rho(\cdot,\cdot)=(\lambda_1+\lambda_2+\lambda_3)\g{\cdot}{\cdot}$ and \model is Einstein.
\end{proof}

We continue our study of an Einstein model taking advantage of the previous results. Our current task is to find out the mixed terms of the curvature, this is, those which involve vectors $x$, $y$, $z$ and $w$. Observe that Case $e)$ in the proof of Lemma~\ref{lemma:einstein} is more general than Cases $a)$, $b)$, $c)$, $d)$. Thus, we assume henceforth that $\lambda_1=\lambda_6$, $\lambda_2=\lambda_5$ and $\lambda_3=\lambda_4$ so all the possible cases are considered at once.

\begin{lemma}\label{lemma:self-dual}
Let \model be a $4$-dimensional Einstein model which satisfies the Raki\'c duality principle. Then \model is self-dual or anti-self-dual.
\end{lemma}
\begin{proof}
Adopt the notation of Lemmas~\ref{lemma:4-dim-basis} and \ref{lemma:einstein} as concerns the basis $\{x,y,z,w\}$ and the corresponding eigenvalues for the Jacobi operator. Moreover, assume that $\lambda_1=\lambda_6$, $\lambda_2=\lambda_5$ and $\lambda_3=\lambda_4$. In summary, we consider a basis $\{x,y,z,w\}$ such that
\[
\begin{array}{c}
\ja{x}y=\lambda_1 y,\qquad \ja{x}z=\lambda_2 z,\qquad \ja{x}w=\lambda_3 w,\\
\ja{y}z=\lambda_3 z,\qquad \ja{y}w=\lambda_2 w,\qquad \ja{z}w=\lambda_1 w.
\end{array}
\]

Recall notation $r_\theta=\cos \theta x+\sin\theta y$ from Lemma~\ref{lemma:einstein} and consider $\theta=\frac{\pi}{6}$ so $s=r_\frac{\pi}6=\frac{\sqrt{3}}2x+\frac12y$. We repeat a previous argument to see that there exist $a$ and $b$, with $a^2+b^2=1$ and $a,b> 0$ ($a$ and $b$ cannot be $0$ as a consequence of the Raki\'c duality principle; if $a<0$ or $b<0$ change $z$ by $-z$ or $w$ by $-w$ to rearrange signs) such that
\[
\ja{s}(az+bw)=\alpha (az+bw)
\]
for $\alpha$ a real number.
Expand the previous expression to get
\[
\begin{array}{rcl}
\alpha(az+bw)&=&\frac34a\lambda_2z+\frac34b\lambda_3w+\frac14a\lambda_3z+\frac14b\lambda_2w\\
\noalign{\smallskip}
&&+\frac{\sqrt{3}}4a\{A(z,x)y+A(z,y)x\}+\frac{\sqrt{3}}4b\{A(w,x)y+A(w,y)x\}.
\end{array}
\]
Take the inner product with $z$ and $w$ to obtain the following equations:
\begin{equation}\label{eq:aut-1}
\begin{array}{rcl}
\alpha \, a&=&\frac34a\lambda_2+\frac14a\lambda_3+\frac{\sqrt{3}}4b\{A(w,x,y,z)+A(w,y,x,z)\},\\
\noalign{\smallskip}
\alpha\,b&=&\frac34b\lambda_3+\frac14b\lambda_2+\frac{\sqrt{3}}4a\{A(z,x,y,w)+A(z,y,x,w)\}.
\end{array}
\end{equation}

Apply the Raki\'c duality principle
to see that $\ja{az+bw}s=\alpha s$
and expand
\[
\begin{array}{rcl}
\alpha(\frac{\sqrt{3}}2x+\frac12y)&=&\frac{\sqrt{3}}2a^2\lambda_2x+\frac12 a^2 \lambda_3y+\frac{\sqrt{3}}2 b^2 \lambda_3x+\frac{1}2 b^2\lambda_2y\\
\noalign{\smallskip}
&&+\frac{\sqrt{3}}2ab\{A(x,z)w+A(x,w)z\}+\frac{1}2ab\{A(y,z)w+A(y,w)z\}.
\end{array}
\]
Now take the inner product with $x$ and $y$ to obtain:
\begin{equation}\label{eq:aut-2}
\begin{array}{rcl}
\frac{\sqrt{3}}2 \alpha&=&\frac{\sqrt{3}}2a^2\lambda_2+\frac{\sqrt{3}}2 b^2 \lambda_3+\frac{1}2ab\{A(y,z,w,x)+A(y,w,z,x)\},\\
\noalign{\smallskip}
\frac{1}2 \alpha&=&\frac12 a^2 \lambda_3+\frac{1}2 b^2\lambda_2+
\frac{\sqrt{3}}2ab\{A(x,z,w,y)+A(x,w,z,y)\}.
\end{array}
\end{equation}

On the one hand, from \eqref{eq:aut-1} we obtain the following relation on the eigenvalues:
\[
\frac{a}b(\alpha  -\frac34\lambda_2-\frac14\lambda_3)=\frac{b}a(\alpha-\frac34\lambda_3-\frac14\lambda_2),
\]
that we can write as
\begin{equation}\label{eq:calc1}
\alpha(a^2-b^2)+(\frac14 b^2-\frac34a^2)\lambda_2+(\frac34 b^2-\frac14 a^2)\lambda_3=0.
\end{equation}
On the other hand, from \eqref{eq:aut-2} we get
\[
\alpha = (\frac{3}2a^2-\frac{1}2 b^2)\lambda_2+(-\frac{1}2 a^2+\frac{3}2 b^2 )\lambda_3.
\]
Now, use that $b^2=1-a^2$ and substitute $\alpha$ in \eqref{eq:calc1}:
\[
(\lambda_2-\lambda_3)(8a^4-8a^2+\frac32)=0
\]
Since we are assuming $a,b>0$, the possible solutions are:
\begin{itemize}
\item $\alpha=\lambda_2 = \lambda_3$,
\item $\alpha=\lambda_2\neq \lambda_3$, $a=\frac{\sqrt{3}}2$ and $b=\frac12$,
\item $\alpha=\lambda_3\neq \lambda_2$, $a=\frac{1}2$ and $b=\frac{\sqrt{3}}2$.
\end{itemize}
Now, substitute in \eqref{eq:aut-1} or \eqref{eq:aut-2} to see that if $\alpha=\lambda_2=\lambda_3$ then $A(x,z,w,y)+A(x,w,z,y)=0$, if $\alpha=\lambda_2\neq \lambda_3$ then $A(x,z,w,y)+A(x,w,z,y)=\lambda_2-\lambda_3$, and if $\alpha=\lambda_3\neq \lambda_2$ then $A(x,z,w,y)+A(x,w,z,y)=\lambda_3-\lambda_2$.

We repeat this argument for $\frac{\sqrt{3}}2x+\frac12z$ and $\frac{\sqrt{3}}2x+\frac12w$ to see that:
\[
\begin{array}{l}
A(x,z,w,y)+A(x,w,z,y)=\pm(\lambda_2-\lambda_3),\\
\noalign{\smallskip}
A(x,w,y,z)+A(x,y,w,z)=\pm(\lambda_3-\lambda_1),\\
\noalign{\smallskip}
A(x,y,z,w)+A(x,z,y,w)=\pm(\lambda_1-\lambda_2).
\end{array}
\]
Changing the sign of a vector in $\{x,y,z,w\}$ if necessary, we can assume without loss of generality that:
\[
\begin{array}{l}
A(x,z,w,y)+A(x,w,z,y)=\lambda_2-\lambda_3,\\
\noalign{\smallskip}
A(x,w,y,z)+A(x,y,w,z)=\lambda_3-\lambda_1,\\
\noalign{\smallskip}
A(x,y,z,w)+A(x,z,y,w)=\lambda_1-\lambda_2.
\end{array}
\]
Now compute
\[
\begin{array}{rcl}
\lambda_2-\lambda_3&=&A(x,z,w,y)+A(x,w,z,y)\\
&=&A(x,z,w,y)-A(w,z,x,y)-A(z,x,w,y)\\
&=&2A(x,z,w,y)-A(w,z,x,y)\\
&=&2A(x,z,w,y)+\lambda_1-\lambda_2-A(z,x,w,y)\\
&=&3A(x,z,w,y)+\lambda_1-\lambda_2
\end{array}
\]
to see that $A(x,z,w,y)=\frac{-\lambda_1+2\lambda_2-\lambda_3}{3}$. Use the previous relations and equations \eqref{eq:sim-curv} to compute the other components of the curvature which involve all the elements of the basis $\{x,y,z,w\}$. Hence the curvature tensor is given by:
\[
\begin{array}{c}
A(x,y,y,x)=A(z,w,w,z)=\lambda_1,\, A(x,z,z,x)=A(y,w,w,y)=\lambda_2,\\
\noalign{\smallskip}
A(x,w,w,x)=A(y,z,z,y)=\lambda_3,\, A(x,z,w,y)=\frac{-\lambda_1+2\lambda_2-\lambda_3}{3}\,,\\
\noalign{\smallskip}
A(x,w,z,y)=\frac{\lambda_1+\lambda_2-2\lambda_3}{3}\,, A(x,y,w,z)=\frac{-2\lambda_1+\lambda_2+\lambda_3}{3}\,.
\end{array}
\]
Recall that $\rho(\cdot,\cdot)=(\lambda_1+\lambda_2+\lambda_3)\langle\cdot,\cdot\rangle$ and $\tau=4(\lambda_1+\lambda_2+\lambda_3)$. The Weyl tensor is given by:
\[
\begin{array}{c}
W(x,y,y,x)=W(z,w,w,z)=\frac{2\lambda_1-\lambda_2-\lambda_3}3,\\
\noalign{\smallskip}
W(x,z,z,x)=W(y,w,w,y)=\frac{-\lambda_1+2\lambda_2-\lambda_3}3,\\
\noalign{\smallskip}
W(x,w,w,x)=W(y,z,z,y)=\frac{-\lambda_1-\lambda_2+2\lambda_3}3,\\
\noalign{\smallskip}
W(x,z,w,y)=\frac{-\lambda_1+2\lambda_2-\lambda_3}{3}, W(x,w,z,y)=\frac{\lambda_1+\lambda_2-2\lambda_3}{3},
W(x,y,w,z)=\frac{-2\lambda_1+\lambda_2+\lambda_3}{3}\,.
\end{array}
\]
Now we see that \model is Osserman by checking that, with the chosen orientation, it is anti-self-dual:
\[
\begin{array}{l}
2W^+_{11}=W_{1212}+W_{3434}+2W_{1234}=\frac{-2\lambda_1+\lambda_2+\lambda_3}3+\frac{-2\lambda_1+\lambda_2+\lambda_3}3+2\frac{2\lambda_1-\lambda_2-\lambda_3}{3}=0,\\
\noalign{\smallskip}
2W^+_{22}=W_{1313}+W_{2424}-2W_{1324}=\frac{\lambda_1-2\lambda_2+\lambda_3}3+\frac{\lambda_1-2\lambda_2+\lambda_3}3-2\frac{\lambda_1-2\lambda_2+\lambda_3}{3}=0,\\
\noalign{\smallskip}
2W^+_{33}=W_{1414}+W_{2323}+2W_{1423}=\frac{\lambda_1+\lambda_2-2\lambda_3}3+\frac{\lambda_1+\lambda_2-2\lambda_3}3+2\frac{-\lambda_1-\lambda_2+2\lambda_3}{3}=0,\\
\noalign{\smallskip}
2W^+_{12}=W_{1213}-W_{1224}+W_{3413}-W_{3424}=0,\\
\noalign{\smallskip}
2W^+_{13}=W_{1214}+W_{1223}+W_{3414}+W_{3423}=0,\\
\noalign{\smallskip}
2W^+_{23}=W_{1314}+W_{1323}-W_{2414}-W_{2423}=0.
\end{array}
\]
\end{proof}

{\it Proof of Theorem~\ref{th:main}}. That (i) implies (ii) was proved in \cite{R} (see \cite{Gilkey-book} for an alternative proof). That (ii) implies (i) is a direct consequence of Theorem~\ref{th:einstein-self-dual} and Lemmas~\ref{lemma:einstein} and \ref{lemma:self-dual}.$\hfill\square$


\begin{thebibliography}{99}
\bibitem{AR} V. Andreji\'c, Z. Raki\'c; On the duality principle in pseudo-Riemannian Osserman manifolds, J. Geom. Phys. {\bf 57} (2007), 2158--2166.

\bibitem{Chi1}
Q. S. Chi;
A curvature characterization  of certain locally rank-one
symmetric spaces, \emph{J. Differential Geom.} \textbf{28} (1988), 187--202.

%\bibitem{Chi2}
%Q. S. Chi; Curvature characterization and classification of rank-one symmetric spaces, {\it Pacific J. Math.} {\bf 150} (1) (1991), 31--42.

\bibitem{GKV}
E. Garc\'{\i}a-R\'{\i}o, D. N. Kupeli,  R. V\'{a}zquez-Lorenzo;
{\it Osserman manifolds in semi-Riemannian geometry}, Lect. Notes
Math. {\bf 1777}, Springer-Verlag, Berlin, Heidelberg, New York,
2002.

\bibitem{Gilkey-book}
P. Gilkey; {\it Geometric Properties of Natural Operators Defined
by the Riemannian Curvature Tensor}, World Scientific Publishing
Co., Inc., River Edge, NJ, 2001.

\bibitem{Gilkey-Swann-Vanhecke}
P. Gilkey, A. Swann, and L. Vanhecke; Isoparametric
geodesic spheres and a conjecture of Osserman concerning the
Jacobi operator, {\it Quart. J. Math. Oxford} {\bf 46} (1995), 299--320.

\bibitem{gilkey-p-osserman}
P. Gilkey;
Algebraic curvature tensors which are $p$-Osserman,
{\it Differential Geom. Appl.} {\bf 14} (2001), no. 3, 297--311.

\bibitem{Ni1}
Y. Nikolayevsky;
Osserman manifolds of dimension $8$, \emph{Manuscripta Math.}
\textbf{115} (2004), 31--53.

\bibitem{Ni2}
Y. Nikolayevsky;
Osserman conjecture in dimension $\neq 8, 16$, \emph{Math. Ann.} \textbf{331} (2005), 505--522.


\bibitem{Osserman}
R. Osserman; Curvature in the eighties, \emph{Amer. Math. Monthly} \textbf{97} (1990), 731--756.

\bibitem{R} Z. Raki\'c; On duality principle in Osserman manifolds, {\it Linear Algebra Appl.} {\bf 296} (1999), 183--189.
\end{thebibliography}
\end{document}